\documentclass{amsart}
\usepackage{amsmath,amssymb}
\usepackage{graphicx,subfigure}
\usepackage{mathtools,slashed}
\newtheorem{theorem}{Theorem}[section]
\newtheorem{proposition}[theorem]{Proposition}

\newtheorem{lemma}[theorem]{Lemma}

\theoremstyle{definition}

\theoremstyle{remark}
\newtheorem{remark}[theorem]{Remark}

\numberwithin{equation}{section}

\newcommand{\al}{\alpha}

\newcommand{\de}{\delta}

\newcommand{\vphi}{\varphi}

\newcommand{\om}{\omega}
\newcommand{\si}{\sigma}

\newcommand{\De}{\Delta}

\newcommand{\La}{\Lambda}

\def\hg{\widehat g}

\newcommand{\tpsi}{\widetilde{\psi}}

\newcommand{\tY}{\widetilde{Y}}

\def\CC{\mathbb{C}}
\def\NN{\mathbb{N}}

\def\RR{\mathbb{R}}
\def\BB{\mathbb{B}}
\def\ZZ{\mathbb{Z}}

\def\TT{\mathbb{T}}

\renewcommand\SS{\mathbb{S}}
\newcommand\minus\backslash

\newcommand\lan\langle
\newcommand\ran\rangle

\newcommand{\e}{{e}}

\DeclareMathOperator\dist{dist}

\renewcommand\leq\leqslant
\renewcommand\geq\geqslant
\newlength{\intwidth}

\newcommand{\Cone}{\mathbb{C}}

\begin{document}

\title[Nodal sets of eigenfunctions of the Dirac operator]{Geometric structures in the nodal sets of eigenfunctions of the Dirac operator}

\author{Francisco Torres de Lizaur}
\address{Instituto de Ciencias Matem\'aticas, Consejo Superior de
  Investigaciones Cient\'\i ficas, 28049 Madrid, Spain}
\email{fj.torres@icmat.es}

\begin{abstract}

We show that, in round spheres of dimension $n\geq3$, for any given collection of codimension 2 smooth submanifolds $\mathfrak{S}:=\{\Sigma_1,...,\Sigma_N\}$ of arbitrarily complicated topology ($N$ being the complex dimension of the spinor bundle), there is always an eigenfunction $\psi=(\psi_1,...,\psi_N)$ of the Dirac operator such that each submanifold $\Sigma_a$, modulo ambient diffeomorphism, is a structurally stable nodal set of the spinor component $\psi_a$. The result holds for any choice of trivialization of the spinor bundle. The emergence of these structures takes place at small scales and sufficiently high energies.

\end{abstract}
\maketitle

\section{Introduction}

Regarding the spectral properties of elliptic operators on a compact manifold, a problem of much physical and mathematical significance is to understand the ultraviolet regime, that is, which patterns emerge as the eigenvalues get larger.

The paradigmatic example of such a pattern is Weyl's law on the growth of the number of eigenvalues of the Laplace operator, which first appeared as a heuristic derivation of the energy distribution of black body radiation (the Rayleigh-Jeans law, at the heart of the ultraviolet catastrophe). A more modern example, to this day the object of intense investigation, is S. T. Yau's conjecture \cite{Yau} on the growth of the Hausdorff measure of the zero sets of eigenfunctions: the total hypersurface measure is expected to be proportional to the square root of the eigenvalue, regardless of the manifold. In the case of real analytic Riemannian metrics, this conjecture was proved by H. Donnelly and C. Fefferman \cite{DF}; in the general case, some very recent breakthroughs have been made by A. Logunov \cite{LO1, LO2} and A. Logunov and E. Malinnikova \cite{LOMA}.
%$c \lambda \leq \cH^{n-1}(\vphi^{-1}_{\lambda}(0))\leq C \lambda$
%Both asymptotic laws share their not being dependent on....

These asymptotic laws suggest a certain universality of the ultraviolet behavior. Indeed, aside from constants, they do not depend on the metric used to define the elliptic operator, nor do they detect the global topology of the underlying manifold. Even if the magnitudes of interest are global, they seem to be controlled only by the small scale behavior. The intuition is that, at sufficiently high energies, the characteristic scales are very small with respect to those relevant to the geometry and topology of the manifold, hence the behavior of eigenfunctions should not be very sensible to these.

Our aim in this paper is to investigate, from the above perspective, the \emph{topology} of the nodal sets of eigenfunctions of \emph{Dirac operators}. 

Dirac operators have gained a central role in geometry, by virtue of their analytic properties (especially the index theorem and the Weitzenb\H{o}ck formula), and how these relate to the geometry and topology of the manifold (see e.g \cite{GRO-LAW, WITT}). What can be learned from their spectral properties has also interested mathematicians and physicists alike \cite{FRI, BAR12, VAFAWITT}. 

On the other hand, the problem of the allowed shapes of zero sets of solutions to elliptic PDEs also has some interesting history. In the case of the Poisson equation, the problem intrigued mathematical physicists already in the nineteenth century, because of their interest in describing the surfaces of constant gravitational or electric potential (for example, to characterize the possible equilibrium shapes of a gravitating fluid). In the case of the Cauchy-Riemann equations, it corresponds to a weakened version of the second Cousin problem: what codimension 2 submanifolds of a complex manifold can be (maybe up to diffeomorphism) the nodal set of a holomorphic function? As first shown by K. Oka \cite{OKA} (in what became both one of the precursors of sheaf theoretical methods in algebraic geometry, and a first hint of M. Gromov's Oka Principle \cite{GRO}), in the case of a Stein manifold, the only obstruction is the obvious one: it must be possible to realize the submanifold in question as the zero set of a continuous, complex-valued function.

In a $n$-dimensional spin manifold $M$, we can formulate an analogous problem for eigenfunctions of the Dirac operator. Recall that eigenfunctions of the Dirac operator are sections of a hermitian vector bundle $S$ of complex rank $r(n)=2^{\lfloor \frac{n}{2}\rfloor}$, called the spinor bundle. In a spin manifold of dimension 3 or higher, the regular zero sets of a spinor are empty (because $2r(n) > n$), so we will focus our attention on the topology of the zero sets of the spinor components. These are complex-valued, so their regular zero sets are, as in the holomorphic case, codimension 2 submanifolds (interestingly, by a result of C. B\H{a}r \cite{BAR2}, critical level sets of a spinor are also of codimension at most 2).

To be more precise, if the spinor bundle is trivial, a choice of trivialization makes any section of $S$ a collection of $r(n)$ complex-valued functions; if the section is an eigenfunction of the Dirac operator, these complex-valued functions are related by a first order partial differential relation, and this relation might impose restrictions on how topologically intricate the zero sets of the functions can get to be.  For example, in $\SS^{3}$, a spinor can be decomposed in two components $(\psi_1, \psi_2)$. If $L_1$ and $L_2$ are two disjoint closed curves, arbitrarily knotted and linked, one might ask whether it is possible to find an eigenfunction of the Dirac operator such that $L_1$ is a nodal set of $\psi_1$, and $L_2$ is a nodal set of $\psi_2$, possibly up to an ambient diffeomorphism (henceforth, by a nodal set of a function we mean a union of connected components of its zero set, not necessarily the whole zero set). 

More generally, given a collection of $r(n)$ codimension 2 submanifolds, $\mathfrak{S}:=\{ \Sigma_a \}_{a=1}^{r(n)}$, we will say that a section of the spinor bundle $\psi$ \emph{realizes} $\mathfrak{S}$ if each submanifold $\Sigma_a$, possibly modulo an ambient diffeomorphism $\Phi$, is a nodal set of the corresponding spinor component $\psi_a$, $\Phi(\Sigma_a)\subset \psi_a^{-1}(0)$.  The main result in this paper is that, in the case of the round $n$-dimensional sphere, any collection can be realized, for any given trivialization, by Dirac eigenfunctions \emph{of high enough energy}:

%The question is whether any such collection of submanifolds $\Sigma$ can be realized by an eigenfunction of the Dirac operator, and this for $\emph{any}$ choice of trivialization.

%----It is easy to see that one can always construct spinors realizing any given structure $\Sigma$.  

%What makes the above problem interesting is the combination of the topological realization requirement with the rigid analytic structure imposed by the Dirac operator. 

\begin{theorem}\label{T.main1}
Let $\mathfrak{S}:=\{ \Sigma_a \}_{a=1}^{r(n)}$ be a collection of $r(n)$ closed, pairwise disjoint, smooth codimension 2 submanifolds in $\SS^n$, for $n\geq 3$. Fix an integer $m\geq1$. For any large enough positive integer $k$, and for any choice of orthonormal basis of Killing spinors trivializing the spinor bundle, there is a $C^{m}$-open set of eigenfunctions of the Dirac operator of eigenvalue $\pm(\frac{n}{2}+k)$ realizing $\mathfrak{S}$.
\end{theorem}

The proof exhibits the interplay between, on one side, the flexibility of such questions in euclidean space (or more generally, for open manifolds); and on the other side, a certain notion of universality at high energies and corresponding small scales that reduces the problem to the euclidean case. 

%The results holds also for negative eigenvalues $-(\frac{n}{2}+k)$ with trivial modifications in the proof, but for ease of notation we will concentrate on positive eigenvalues. 

Flexibility in euclidean space is captured by the following result, whose proof will be given in section \ref{S.euclidean}. If we call $D_0$ the standard Dirac operator on $\RR^{n}$, it reads:

\begin{theorem}\label{T.euclidean}
Fix an integer $m\geq 1$, and an arbitrarily small real number $\epsilon>0$. Inside the unit ball $\BB^{n} \subset \RR^{n}$, consider a collection $\mathfrak{S}:=\{ \Sigma_a \}_{a=1}^{r(n)}$ of $r(n)$ closed, pairwise disjoint, smooth codimension 2 submanifolds.  For any given $\lambda \in \RR$, there is a $\CC^{r(n)}$-valued function $\phi:=(\phi_1,...,\phi_{r(n)})$, satisfying the Dirac equation $D_0 \phi=\lambda \phi$ on $\RR^{n}$, and a diffeomorphism $\Phi_0: \BB^{n} \rightarrow \BB^{n}$ satisfying $||\Phi_0- Id||_{C^{m}}\leq \epsilon$, such that $\Phi_0(\Sigma_a)$ is a nodal set of $\phi_a$. Further, these nodal sets are structurally stable.
\end{theorem} 

The nodal sets being structurally stable means that any other $\CC^{r(n)}$-valued function $\vphi:=(\vphi_1,...,\vphi_{r(n)})$ that is close enough to $\phi=(\phi_1,...,\phi_{r(n)})$ in the $C^{m}$ topology will have the same collection of nodal sets, modulo a diffeomorphism arbitrarily close to the identity (more precisely, once one choses how close to the identity the diffeomorphism is to be, there exists a bound on how close to $\phi$ a function $\vphi$ needs to be). This stability is a consequence of the fact that, by our construction, the derivatives $d_x \phi_a$ will have full rank for all $x\in \Phi(\Sigma_a)$ (the relation between the full rank of the derivatives and the stability of the structures is often called Thom's isotopy lemma, see e.g \cite{AR}). 

In ~\cite{EP}, A. Enciso and D. Peralta-Salas introduced a general strategy to tackle similar realization questions for level sets of solutions of second order elliptic PDEs in euclidean space. It is based on finding local solutions with prescribed 1-jet on the submanifold one wants to realize (which can be done by means of Cauchy-Kovalevskaya theorem, or solving a well chosen boundary value problem), and then extending these local solutions to global ones by means of a Runge-type approximation theorem. That we have prescribed the first derivative at the submanifold becomes crucial in this last step, to ensure the structural stability of the level set. 

%It is for this last step that the prescription of the first derivative is crucial

The proof of our euclidean result adapts this strategy to the case of the Dirac operator, but this requires some further considerations. The main difficulty is that now we can only prescribe the 0-jet on the local Cauchy problem, because the Dirac operator is of order one, so structural stability does not come out automatically and a more involved construction is needed.

On compact manifolds, however, the whole strategy above is guaranteed to fail from the start. The reason for this failure is of a fundamental nature: on a compact manifold, one cannot use Runge-type approximations without creating singularities in the global solutions.

In contrast, the idea of the proof of Theorem \ref{T.main1} is to regain some flexibility by exploiting the increasing dimension of the space of eigenfunctions of the Dirac operator. This is analogous to the strategy used in \cite{EPT} in the context of invariant sets of eigenfunctions of the curl operator. Specifically, we prove that for any given eigenfunction $\phi$ of the euclidean Dirac operator of eigenvalue 1, there are many Dirac spinors in the sphere of high enough eigenvalue $\frac{n}{2}+k$ whose behavior (in the sense of $C^{m}$ norm) in a ball of radius $k^{-1}$ is very close to the behavior of $\phi$ in the unit ball.

This implies that any property that eigenfunctions of the euclidean Dirac operator exhibit on compact sets is also exhibited at small scales by high energy eigenfunctions of the Dirac operator on the sphere, provided such property is robust under suitably small perturbations. Theorem \ref{T.euclidean} precisely ensures that the realization of any arbitrarily complicated collection of submanifolds $\mathfrak{S}$ is a property of this kind.

The proof of the theorem yields as well a rather precise understanding of the ambient diffeomorphism $\Phi$ through which the structure $\mathfrak{S}$ is realized. It basically consists on rescaling an open subset containing $\mathfrak{S}$ so that it gets inside a ball of radius proportional to $k^{-1}$. In principle, we have no control on the nodal sets outside the small ball. Nonetheless, the analysis can be refined so that one is able to construct eigenfunctions with prescribed nodal sets inside a given number (as large as we want) of small enough balls.

The paper is organized as follows. In Section~\ref{S.main} we will
prove Theorem \ref{T.main1}, assuming the euclidean realization theorem (Theorem \ref{T.euclidean}) and the key inverse localization result (Theorem~\ref{T.approx1}). The inverse localization result is then proved in Section~\ref{S.approx1}, with the proofs of some additional propositions being postponed to Section ~\ref{S.spharm1}. The euclidean realization theorem is proved in Section~\ref{S.euclidean}. In Section~\ref{S.lens} we adapt the previous results to prescribe the nodal sets at multiple regions at once. We conclude with some comments about the analog of Theorem \ref{T.main1} in the torus case, which can be proved with minor adaptations of the ideas in this paper.

\section{Proof of the main theorem}
\label{S.main}

To set the stage, consider the sphere $\SS^{n}$ ($n\geq3$) endowed with the round metric $g$ of constant sectional curvature 1. We will denote by $S$ the spinor bundle, whose sections $\psi: \SS^{n}\rightarrow S$ we call spinors. The canonical covariant derivative on $S$ associated with the Levi-Civita connection on $T\SS^{n}$ will be denoted by $\nabla^{S}$. The Dirac operator $D$ can be written locally as
\[
D=\sum_{i=1}^{n} \rho(e_{i}) \nabla_{e_{i}}^{S},
\]
where $\rho$ is the Clifford multiplication map, $\rho: T\SS^{n}\rightarrow \text{End}(S)$, and $\{e_{i}\}^n_{i=1}$ is an orthonormal basis of $T\SS^n$.

The spectrum of the Dirac operator on $\SS^{n}$ is well studied (see e.g \cite{BAR1}). In particular, eigenvalues are of the form $\pm (\frac{n}{2}+k)$, for $k\in \NN$, and the linear space of eigenfunctions of fixed $k$ has complex dimension

\[
D(n,k)=r(n)\binom{n+k-1}{k}.
\]
%These results can be obtained by representation theoretic methods, since $\SS^{n}=Spin(n+1)/Spin(n)$, or in a more direct manner using Killing spinors (see e.g \cite{BAR1}).
A very explicit characterization of the spinor bundle of the $n$-dimensional sphere can be provided in terms of \emph{Killing spinors}. Killing spinors are solutions of the equation
\[
\nabla^{S} \chi= \lambda \rho(\cdot) \chi, 
\]
for some fixed constant $\lambda$, called the Killing number of the Killing spinor.

Note that a Killing spinor is either identically zero, or it has no zeroes at all. Indeed, the Killing equation implies, on one hand, that a Killing spinor is also an eigenfunction of the Dirac operator, so that it satisfies the strong unique continuation property; and on the other hand, that in a point where a Killing spinor vanishes, all of its derivatives vanish too.

%are called Killing spinors of Killing number $\lambda$. 

It is well known (see e.g \cite{BAR1}) that for $\lambda=\pm \frac{1}{2}$ one can define a global orthonormal frame $\{ \chi_{a} \}_{a=1}^{r(n)}$ of $S$ consisting on $\lambda$-Killing spinors (the linear space of Killing spinors of Killing number $\pm \frac{1}{2}$ is precisely of complex dimension $D(n,0)=r(n)$). Thus, any section $\psi$ of $S$ can be specified by the choice of $r(n)$ complex functions $\psi_a : \SS^{n}\rightarrow \Cone$,

\[
\psi=\sum_{a=1}^{r(n)} \psi_a \chi_a.
\]

Let us now introduce the notations needed to describe the small scale behavior of spinors. Fix an arbitrary point $p_0\in\SS^{n}$ and take a patch of normal geodesic coordinates
$\Psi:\BB\to B$ centered at $p_0$. We will denote by $B_\rho$ (resp.~$\BB_\rho$) the
ball in~$\RR^n$ (resp.\ the geodesic ball
in~$\SS^{n}$) of 
radius~$\rho$ and centered at the origin (resp.\ at $p_0$). We omit the subscript when $\rho=1$. Via the diffeomorphism $\Psi$ one defines a map on the spinor bundle
\[
\widehat{\Psi}_{*}: S |_{\BB} \longrightarrow B\times\CC^{r(n)};
\]

and a spinor field $\psi$ will be expressed in local coordinates $x$ on $B$ as 
\[
\widehat{\Psi}_{*}\psi(x) =\sum_{a=1}^{r(n)} \widehat{\psi}_{a}(x)\, \xi_a\,,
\]
where $\{\xi_a\}_{a=1}^{r(n)}$ is the standard basis on $\CC^{r(n)}$. Throughout the rest of the paper, unless we explicitly say otherwise, we assume the choice of a basis of Killing spinors $\chi_a$, and we choose the basis $\xi_a$ accordingly, that is, with $\widehat{\Psi}_{*}\chi_a(0) =\xi_a$. 

The small scale behavior of the spinor $\psi$ on a ball of radius $k^{-1}$ around $p_0$  is thus captured by the rescaled field
$$
\widehat{\Psi}_{*}\psi \Big(\frac{\cdot}{k}\Big):=\sum_{a=1}^{r(n)}
\widehat{\psi}_{a}\Big(\frac{\cdot}{k}\Big)\, \xi_a\,.
$$

Note that, in general, the expression of the $a$-th spinor component $\psi_a$ in normal coordinates, $\psi_{a}\circ\Psi^{-1}$, does correspond exactly to the euclidean component $\widehat{\psi}_{a}$. Still, one has
\begin{equation}\label{eq.loc}
\psi_a \circ \Psi^{-1}\Big(\frac{\cdot} {k}\Big)=\widehat{\psi}_{a}\Big(\frac{\cdot}{k}\Big)+\sum_b A_{ab}\Big(\frac{\cdot}{k}\Big)\widehat{\psi}_{b}\Big(\frac{\cdot}{k}\Big)
\end{equation}
with smooth functions $A_{ab}$ verifying 
\[
\bigg\|A_{ab}\Big(\frac{\cdot}{k}\Big)\bigg\|_{C^m(B)}\leq\frac{C_m}{k},
\]
with constants $C_m$ not depending on $k$.
Finally, we will denote the standard Dirac operator on $\RR^{n}$ by $D_0$.

We are now ready to state the following approximation theorem, which is the main ingredient in the proof. It makes precise the statement in the introduction to the effect that the increasing degeneracy of the spectrum of the Dirac operator introduces the exact amount of flexibility we need in our problem. For the sake of concreteness, we will concentrate henceforth on the case of positive eigenvalues, the negative case being completely analogous:

\begin{theorem}\label{T.approx1}
Let $\phi:=(\phi_1,...,\phi_{r(n)})$ be a $\CC^{r(n)}$ valued function in $\RR^n$, satisfying the Dirac equation $D_0 \phi=\phi$. Fix an integer $m\geq1$ and a positive constant $\delta$. For any large enough positive integer $k$, there is an eigenfunction $\psi$ of the Dirac operator $D$ on $\SS^{n}$ of eigenvalue $(\frac{n}{2}+k)$ such that
\begin{equation*}%\label{prop2e1}
\bigg\|\phi-\widehat{\Psi}_{*}\psi \Big(\frac\cdot k\Big)\bigg\|_{C^{m}(B)}<\delta\,.
\end{equation*}
\end{theorem}

Note that it is the converse statement which is trivially true, not only in spheres but in any spin manifold: an eigenfunction of the Dirac operator of high enough eigenvalue $\lambda$ always behaves, at scales of order $\lambda^{-1}$, as an eigenfunction of the euclidean Dirac operator of eigenvalue 1; however, this is not very revealing, since in principle we have no information whatsoever as to which particular eigenfunction we are converging to, and its properties. 

By contrast, the above theorem ensures that any property that can be exhibited by solutions to the euclidean Dirac equation $D_0 \phi=\phi$ on Euclidean space is also exhibited, at small scales, by eigenfunctions of $D$ of high enough eigenvalue, provided such property is robust under $C^{m}$ perturbations.

Now, let $\Phi'$ be a diffeomorphism of $\SS^{n}$ mapping
the collection $\{ \Sigma_a \}$ into the ball $\BB_{1/2}$, and the ball $\BB_{1/2}$
into itself. Consider the collection of codimension two submanifolds $\{ \Sigma'_{a}\}$
in $B_{1/2}$, defined as
$$
\Sigma'_a:=(\Psi\circ \Phi')(\Sigma_a)\,.
$$
The euclidean realization Theorem \ref{T.euclidean} yields a $\CC^{r(n)}$-valued function $\phi=(\phi_1,...,\phi_r(n))$ and a diffeomorphism $\Phi_0$, very close to simply being the identity, such that $\Phi_0(\Sigma'_a) \subset \phi_{a}^{-1}(0)$. 

Besides, Theorem \ref{T.approx1} allows us to find, for any large enough integer $k$, a spinor $\psi$ verifying $D\psi=(\frac{n}{2}+k)\psi$ and approximating $\phi$ in the $C^{m}(B)$ norm as much as we want. In particular, we can choose $k$ so that each component $\psi_a$ of the spinor approximates $\phi_a$ as much as we want. In view of equation \ref{eq.loc}, we have
\[
\bigg\|\phi_a-\psi_a \circ \Psi^{-1}\Big(\frac{\cdot} {k}\Big)\bigg\|_{C^{m}(B)}\leq \bigg\|\phi_a-\widehat{\psi}_{a}\Big(\frac{\cdot} {k}\Big)\bigg\|_{C^{m}(B)}+\frac{C}{k}\bigg\|\phi-\widehat{\Psi}_{*} \psi\Big(\frac{\cdot} {k}\Big)\bigg\|_{C^{m}(B)}+\frac{C}{k}||\phi||_{C^{m}(B)},
\]

so, choosing $k$ big enough, we get 
$$
\bigg\|\phi_a-\psi_a \circ \Psi^{-1}\Big(\frac{\cdot} {k}\Big)\bigg\|_{C^{m}(B)}<\delta\,
$$
for $\delta>0$ as small as we want.

The structural stability property ensures then the existence of a diffeomorphism $\Phi_1: \RR^{n} \rightarrow \RR^{n}$ very close to the identity, such that $\Phi_1(\Phi_0(\Sigma'_a)) \subset \psi_a\circ \Psi^{-1}(\frac{\cdot} {k})$. Therefore, each $\psi_a$ has the corresponding submanifold $\Psi^{-1}(\Phi_1(\Phi_0(\Sigma'_a))$ as a nodal set, and we conclude that $\psi$ realizes $\mathfrak{S}$. Every other eigenfunction close enough to $\psi$ will also realize $\mathfrak{S}$, through a slightly perturbed diffeomorphism.

\section{Inverse localization: proof of Theorem~\ref{T.approx1}}
\label{S.approx1}

On any spin manifold, the Weitzenb\H{o}ck identity relates two natural second order elliptic operators arising from the spinor connection: the square of the Dirac operator $D^{2}$, and the covariant laplacian $\Delta_S:=\nabla^{S*} \nabla^{S}$ (here $\nabla^{S*}$ is the $L^2$ adjoint of the covariant derivative on the spinor bundle). It reads

\begin{equation}\label{wetz1}
D^{2}=\Delta_S+ \frac{s}{4}\, , 
\end{equation}

where $s$ is the scalar curvature. 

In the $n$-sphere, the Weitzenb\H{o}ck formula provides a link between Dirac spinors and spherical harmonics. To see this, first consider the twisted connection
\[
\nabla^{\chi}:=\nabla^{S}+\frac{1}{2}\rho(\cdot)
\]
whose covariantly constant sections are precisely the Killing spinors of Killing number $-\frac{1}{2}$. The corresponding laplacian $\Delta_{\chi}:=\nabla^{\chi*} \nabla^{\chi}$ has a transparent interpretation when written in an orthonormal frame $\{\chi_a\}_{a=1}^{r(n)}$ of $-\frac{1}{2}$-Killing spinors: a straightforward computation yields

\begin{equation}\label{diracharm}
\Delta_{\chi} \psi=\sum_{a=1}^{r(n)} (\Delta_{\SS^{n}} \psi_a) \chi_a\,,
\end{equation}

where $\Delta_{\SS^{n}}$ is the Laplace-Beltrami operator acting on functions.

On the other hand, consider the twisted Dirac operator
 $$
 \slashed{D}:=D-\frac{1}{2}.
 $$ 
From the Weitzenb\H{o}ck formula we obtain that the couple $\slashed{D}^{2}$ and $\Delta_{\chi}$ satisfies the identity
\begin{equation}\label{wetz2}
\slashed{D}^{2}=\Delta_\chi+ \frac{(n-1)^2}{4}\,,
\end{equation}
where we have used that the scalar curvature of the round $n$-sphere is $n(n-1)$.
Now, let $\psi$ be an eigenfuntion of $D$ of eigenvalue $\frac{n}{2}+k$. It verifies
\begin{equation}\label{wetz3}
\Big(D-\frac{1}{2}\Big)^{2}\psi=\Big(\frac{n-1}{2}+k\Big)^2 \psi = k(n+k-1) \psi + \frac{(n-1)^2}{4} \psi, 
\end{equation}
hence equation \ref{wetz2} allows us to conclude that 
\[
 \Delta_\chi \psi= k(n+k-1) \psi
\]
and so, in view of equation \ref{diracharm}, the components $\psi_a$ of $\psi$ are complex spherical harmonics of eigenvalue $k(n+k-1)$. 

The following proposition captures the small scale behavior of spherical harmonics:

\begin{proposition}\label{P.spharm1}
Let $\phi$ be a complex-valued function in $\RR^n$, satisfying
$\Delta \phi+\phi=0$. Fix a positive integer $m$ and positive constant $\de'$. For any large enough integer $k$, there is a complex-valued spherical harmonic $Y$ on
$\SS^n$ with energy $k(n+k-1)$ such that
\begin{equation*}%\label{prop2e1}
\bigg\|\phi-Y\circ \Psi^{-1}\Big(\frac\cdot k\Big)\bigg\|_{C^{m+2}(B)}\leq\de'\,.
\end{equation*}
\end{proposition}

The proof, which is a direct generalization to any dimension of the procedure in \cite{EPT}, is given in Section \ref{S.spharm1}. Here and in what follows, by $\Delta$ we denote the euclidean negative Laplacian, $\Delta=\sum \partial^{2}_{\mu}$.

%This proposition yields an analogous of Theorem \ref{T.approx1} for eigenfunctions $\tpsi$ of the operator $\slashed{D}^2$. 

Now, let $\phi:=(\phi_1,...,\phi_r(n))$ be a $\CC^{r(n)}$-valued function on $\RR^{n}$ satisfying $D_0 \phi=\phi$. Since $D_{0}^2=-\Delta$, the complex functions $\phi_a$ satisfy the Helmholtz equation $\Delta \phi_a + \phi_a=0$. From Proposition \ref{P.spharm1}, we obtain a collection of spherical harmonics $Y_a$, with $a=1,...,r(n)$, locally approximating the corresponding functions $\phi_a$. The spinor $\tpsi$ defined as
\[
\tpsi_a:=Y_a
\]
%where $Y_a$ is obtained by appliying to $\phi_a$ the above proposition, that is, it verifies
%\[
%\bigg\|\phi-Y_a\circ \Psi^{-1}\Big(\frac\cdot k\Big)\bigg\|_{C^{m+2}(B)}\leq\de'\,.
%\] 
is an eigenfunction of the operator $\slashed{D}^2$ of eigenvalue $k(n+k-1)+\frac{(n-1)^2}{4}$ (by equation \ref{wetz2}). It satisfies
%By equation \ref{wetz2}, the spinors $\tpsi$ are eigenfunctions of $\slashed{D}^2$ of eigenvalue $k(n+k-1)+\frac{(n-1)^2}{4}$; by construction, they satisfy
\[
\bigg\|\phi-\widehat{\Psi}_{*}\tpsi \Big(\frac\cdot k\Big)\bigg\|_{C^{m+2}(B)}\leq\delta'+Ck^{-1}||\phi||_{C^{m+2}(B)},
\]
where the last term on the right hand side of the inequality comes from the mismatch between euclidean and spherical trivializations (as in equation \ref{eq.loc}).

Since $\phi$ is fixed, by choosing $\delta'$ in Proposition \ref{P.spharm1} small enough and $k$ large enough, we get that $\tpsi$ satisfies the bound 
\[
\bigg\|\phi-\widehat{\Psi}_{*}\tpsi \Big(\frac\cdot k\Big)\bigg\|_{C^{m+2}(B)}<\delta.
\] 
for a given $\delta$ as small as we want. 

Given such a $\tpsi$, the following lemma provides us with an eigenfunction of $D$ that still approximates the euclidean spinor $\phi$:

\begin{lemma} \label{trick}
Let $\phi$ and $\tpsi$ be as above. The spinor $\psi$ defined as
\begin{equation*}%\label{prop2e1}
\psi:=\frac{\slashed{D}(\slashed{D}\tpsi+(\frac{n-1}{2}+k)\tpsi)}{2(\frac{n-1}{2}+k)^2}
\end{equation*}
is an eigenfunction of the Dirac operator $D$ of eigenvalue $\frac{n}{2}+k$, satisfying
\[
\bigg\|\phi-\widehat{\Psi}_{*}\psi \Big(\frac\cdot k\Big)\bigg\|_{C^{m}(B)}<C\delta+C k^{-1} ||\phi||_{C^{m+2}(B)}.
\]
with constants $C$ not depending on $k$.
\end{lemma}

\begin{proof} In view of equation \ref{wetz3} , one can easily check that $\psi$ satisfies the identity
\[
 \slashed{D} \psi = \Big(\frac{n-1}{2}+k\Big) \psi,
\]
and since $\slashed{D}=D-\frac{1}{2}$, the spinor $\psi$ is an eigenfunction of $D$ with eigenvalue $(\frac{n}{2}+k)$. With regards to the $C^{m}$ bound, first let us note that we have
\[
\widehat{\Psi}_{*}(\slashed{D} \tpsi)=k\Big(D_0 \widehat{\Psi}_{*} \tpsi+\frac{G_1}{k} \partial \widehat{\Psi}_{*} \tpsi+\frac{G_2}{k}\widehat{\Psi}_{*} \tpsi\Big)
\]
\[
\widehat{\Psi}_{*}(\slashed{D}^2 \tpsi)=k\Big(D_0^2 \widehat{\Psi}_{*} \tpsi+\frac{G_3}{k} \partial^{2} \widehat{\Psi}_{*} \tpsi +\frac{G_4}{k} \partial \widehat{\Psi}_{*} \tpsi+\frac{G_5}{k^2}\widehat{\Psi}_{*} \tpsi\Big)
\]
where the $G_i$ are smooth matrix-valued functions satisfying the uniform bounds
\[
||G_i||_{C^m(B)}\leq{C_m}
\]
and $\partial \widehat{\Psi}_{*}\tpsi$ (resp. $\partial^{2}\widehat{\Psi}_{*}\tpsi$)  is a matrix whose entries are first (resp. second) order derivatives of the components of $\widehat{\Psi}_{*}\tpsi$.

Further, since $2\phi=D_0(D_0+1)\phi$, we have

\begin{align}\label{finalestimation}
\bigg\|\phi-\widehat{\Psi}_{*}\psi\Big(\frac{\cdot}{k}\Big)\bigg\|_{C^m(B)}&\leq\bigg\|\frac12(D_0^2+D_0)(\phi-\widehat{\Psi}_{*}\tpsi\Big(\frac{\cdot}{k}\Big) )\bigg\|_{C^m(B)}+\frac{C}{k}\bigg\|\widehat{\Psi}_{*}\tpsi\Big(\frac{\cdot}{k}\Big)\bigg\|_{C^{m+2}(B)}\notag
  \\&\leq C\bigg\|\phi-\widehat{\Psi}_{*}\tpsi\Big(\frac{\cdot}{k}\Big)\bigg\|_{C^{m+2}(B)}+\frac{C}{k}\bigg\|\phi-\widehat{\Psi}_{*}\tpsi\Big(\frac{\cdot}{k}\Big)\bigg\|_{C^{m+2}(B)}\notag\\
&\qquad \qquad \qquad \qquad \quad+\frac{C}{k}\|\phi\|_{C^{m+2}(B)}\,\leq C \delta+ \frac{C}{k}.  
\end{align}
 and the lemma follows. \end{proof}
 
Finally, Theorem \ref{T.approx1} follows upon choosing $\delta$ as small and correspondingly $k$ as large as needed.

\section{Inverse localization for complex spherical harmonics: proof of Proposition~\ref{P.spharm1}}
\label{S.spharm1}

We show in this section that, given any complex-valued function $\phi$ in $\RR^n$ satisfying $\Delta \phi + \phi=0$, there exists a complex-valued spherical harmonic $Y$ in $\SS^n$, satisfying
$\Delta_{\SS^{n}} Y=k(n+k-1) Y$, whose behavior in a small ball of radius $k^{-1}$ is very
close to that of $\phi$ in the euclidean unit ball $B$. 

We will proceed in two successive approximation steps. First, we will approximate the function $\phi(x)$ in $B$ by a function $\vphi(x)$ that consists on a finite sum of terms of the form 
\[
\frac{1}{|x-x_j|^{\frac{n}{2}-1}}J_{\frac{n}{2}-1}(|x-x_j|) 
\]
with $x_j\in\RR^n$, $j=1,...,N$, for $N$ big enough  (Proposition~\ref{P.Bessel1}). In the second step, we show that there is a spherical harmonic $Y$ in $\SS^n$ of
energy $k(n+k-1)$ which, when considered in a ball of radius $k^{-1}$ with coordinates rescaled to the euclidean ball of radius 1, approximates the function $\vphi$, provided that $k$ is large enough. We recall that the dimension of the space of spherical harmonics on the $n$-sphere of eigenvalue $k(n+k-1)$ is given by 
\[
d(k, n):=\binom{n+k-1}{k} \frac{n+2k-1}{n+k-1}.
\]

\begin{proposition}\label{P.Bessel1}
Given any $\de>0$, there is a finite radius $R$ and finitely many complex-valued
constants $\{c_j\}_{j=1}^N$ and points $\{x_j\}_{j=1}^N\subset B_R$ such that the
function
\[
\vphi:=\sum_{j=1}^N c_j\, \frac{1}{|x-x_j|^{\frac{n}{2}-1}} J_{\frac{n}{2}-1}(|x-x_j|)
\]
approximates the function~$\phi$ in the ball $B$ as
\[
\|\phi-\vphi\|_{C^{m+2}(B)}<\de\,.
\]
\end{proposition}

%\emph{Proof of Proposition \ref{P.Bessel1}}
\begin{proof}

First, we notice that, being a solution of the Helmholtz equation, the complex function $\phi$ can be written in the
ball $B_2$ as an expansion
\begin{equation}\label{fbser}
\phi=\sum_{l=0}^\infty\sum_{j=1}^{d(n-1, l)} b_{lj}\,j_{l}(r)\, Y_{lj}(\om),
\end{equation}
where $r:=|x|\in\RR^+$ and $\om:=x/r\in\SS^{n-1}$, $Y_{lj}$ are a basis of spherical harmonics of eigenvalue $l(l+n-2)$, $j_l$ are $n$-dimensional hyperspherical Bessel functions and $b_{lj}\in\Cone$ are complex coefficients. 

The series ~\ref{fbser} is convergent in the $L^2$ sense, which means that, for any $\de'>0$, we can truncate the sum at some integer $L$
\begin{equation}\label{Lrvb}
\phi_1:=\sum_{l=0}^{L}\sum_{j=1}^{d(n-1, l)} b_{lj}\, j_{l}(r)\, Y_{lj}(\om)
\end{equation}
so that it approximates $\phi$ as
\begin{equation}\label{L2b}
\|\phi_1-\phi\|_{L^2(B_2)}<\de'\,.
\end{equation}

The function $\phi_1$ decays as $|\phi_1(x)|\leq C/|x|^{\frac{n-1}{2}}$ for large enough $|x|$ (because of the decay properties of the spherical Bessel functions). Hence, Herglotz's
theorem (see e.g.~\cite[Theorem 7.1.27]{LAX}) ensures that we can write
\begin{equation}
\phi_1(x)=\int_{\SS^{n-1}}f_1(\xi)\, e^{ix\cdot\xi}\, d\si(\xi)\,,
\end{equation}
where $d\si$ is the area measure on $\SS^{n-1}:=\{\xi\in\RR^n:|\xi|=1\}$ and $f_1$ is a complex-valued function in $L^2(\SS^{n-1})$.

We now choose a smooth function $f_2$ approximating $f_1$ as
\[
\|f_1-f_2\|_{L^2(\SS^{n-1})}<\de'\,,
\]
which is always possible since smooth functions are dense in $L^2(\SS^{n-1})$. The function defined as the inverse Fourier transform of $f_2$, 
\begin{equation}
\phi_2(x):=\int_{\SS^{n-1}}f_2(\xi)\,e^{ix\cdot\xi}\, d\si(\xi)\,,
\end{equation}
approximates $\phi_1$ uniformly: by the
Cauchy--Schwarz inequality, we get
\begin{align}
|\phi_2(x)-\phi_1(x)|=\bigg|\int_{\SS^{n-1}}(f_2(\xi)-f_1(\xi))\,e^{ix\cdot\xi}\,
  d\si(\xi)\bigg|\leq C\|f_2-f_1\|_{L^2(\SS^{n-1})}<C\de'\, \label{eqv2v1}
\end{align}
for any $x\in\RR^n$.

We now want to approximate the function $f_2$ by a trigonometric polynomial: for any given $\de'$, we are going to find a radius $R>0$ and
finitely many points
$\{x_j\}_{j=1}^N\subset B_R$ and constants $\{c_j\}_{n=1}^N\subset\Cone$ such that the smooth function in $\RR^n$  
$$
f(\xi):=\frac{1}{(2\pi)^{\frac{n}{2}}}\sum_{j=1}^N c_j \, e^{-ix_j\xi},
$$
when restricted to the unit sphere, approximates $f_2$ in the $C^0$ norm,
\begin{equation}\label{eqqma}
\|f-f_2\|_{C^0(\SS^{n-1})}<\de'\,.
\end{equation}

In order to do so, we begin by extending $f_2$ to a smooth function $g:\RR^n\to \Cone$ with compact support,
$$
g(\xi):=\chi(|\xi|)\, f_2\bigg(\frac{\xi}{|\xi|}\bigg)\,,
$$
where $\chi(s)$ is a smooth bump function, being $1$ when, for example,
$|s-1|<\frac14$, and vanishing for $|s-1|>\frac12$. The Fourier
transform $\hg$ of $g$ is Schwartz, so it is easy to see that, outside some ball $B_R$, the $L^1$~norm of $\hg$ is very small,
\[
\int_{\RR^n\backslash B_R}|\hg(x)|\, dx<\de'\,,
\]
and therefore we get a very good approximation of $g$ by just considering its Fourier representation with frequencies within the ball $B_R$, that is,
\begin{equation}\label{Rlarge}
\sup_{\xi\in\RR^n}\bigg|g(\xi)-\int_{B_R}\hg(x)\, e^{-ix\cdot\xi}\, dx\bigg|<\de'\,.
\end{equation}

Next, we can approximate the integral
$$
\int_{B_R}\hg(x)\,e^{-ix\cdot\xi}\, dx
$$
by the sum
\begin{equation}\label{fcn}
f(\xi):=\frac{1}{(2\pi)^{\frac{n}{2}}}\sum_{j=1}^N c_j \, e^{-ix_j\cdot \xi}
\end{equation}
with constants $c_j\in\Cone$ and points $x_j\in B_R$, so that we have the bound
\begin{equation}\label{eqdisc}
\sup_{\xi\in \SS^{n-1}}\bigg|\int_{B_R}\hg(x)\, e^{-ix\cdot\xi}\, dx-f(\xi)\bigg|<\de'\,.
\end{equation}

To see this, consider a covering of the ball $B_R$ by closed sets
$\{U_j\}_{j=1}^N$, with piecewise smooth boundaries and pairwise
disjoint interiors, and
whose diameters do not exceed $\de''$. Since the
function~$e^{-ix\cdot\xi}\, \hg(x)$ is smooth, we have that for each $x,y\in U_j$
\[
\sup_{\xi\in\SS^{n-1}}\big|\hg(x)\, e^{-ix\cdot\xi}-\hg(y)\, e^{-i y\cdot\xi}|< C\de''\,,
\]
with the constant $C$ depending on~$\hg$ (and therefore on~$\de'$) but
not on $\de''$. If $x_j$ is any point
in~$U_j$ and we set $c_j:=(2\pi)^{\frac{n}{2}}\hg(x_j)\,|U_j|$ in~\eqref{fcn}, we get
\begin{align*}
\sup_{\xi\in\SS^{n-1}}\bigg|\int_{B_R}\hg(x)\, e^{-ix\cdot\xi}\,
  dx-f(\xi)\bigg|&\leq \sum_{j=1}^N\int_{U_j} \sup_{\xi\in\SS^{n-1}}\big|\hg(x)\,\e^{-i
                   x\cdot\xi}-\hg(x_j)\, e^{-ix_j\cdot\xi}\big|\, dx\\
&\leq C\de''\,,
\end{align*}
with $C$ depending again on $\de'$ and $R$ but not on $\de''$ or $N$. By taking ~$\de''$ so that $C\de''<\de'$, the estimate~\eqref{eqdisc} follows.

Now, in view of ~\eqref{Rlarge} and~\eqref{eqdisc}, one has
\begin{equation*}
\|f-g\|_{C^0(\SS^{n-1})}<C\de'\,,
\end{equation*}
where $C$ does not depend on~$\de'$. The estimate~\eqref{eqqma} follows upon noticing that the function~$f_2$ is the restriction
to~$\SS^{n-1}$ of the function~$g$. 

To conclude, set
\begin{equation*}
\vphi(x):=\int_{\SS^{n-1}}f(\xi)\, e^{ix\cdot\xi}\, d\si(\xi)=\sum_{j=1}^N
\frac{1}{(2\pi)^{\frac{n}{2}}} c_{j}\int_{\SS^{n-1}}e^{i(x-x_{j})\cdot \xi}\,d\si(\xi)=\sum_{j=1}^N
c_{j}\, \frac{1}{|x-x_j|^{\frac{n}{2}-1}}J_{\frac{n}{2}-1}(|x-x_j|)\,,
\end{equation*}
from Equation \eqref{eqqma} we infer that 
\begin{equation*}
\|\vphi-\phi_2\|_{C^0(\RR^n)}\leq \int_{\SS^{n-1}}|f(\xi)-f_2(\xi)|\, d\si(\xi)<C\de'\,,
\end{equation*}
and from Equations \eqref{L2b} and \eqref{eqv2v1} we get the
$L^2$ estimate
\begin{align}\label{cas}
\|\phi-\vphi\|_{L^2(B_2)}\leq C\|\vphi-\phi_2\|_{C^0(\RR^n)}+C\|\phi_2-\phi_1\|_{C^0(\RR^n)}+\|\phi_1-\phi\|_{L^2(B_2)}<C\de'\,.
\end{align}
Furthermore, both $\vphi$ and $\phi$ satisfy the Helmholtz equation in $\RR^n$ (note that the Fourier transform of~$\vphi$ is supported on~$\SS^{n-1}$), so by standard elliptic estimates we have
\begin{equation*}
 \|\phi-\vphi\|_{C^{m+2}(B)}\leq C\|\phi-\vphi\|_{L^2(B_2)}< C\delta'\,,
\end{equation*}
and taking $\delta'$ small enough so that $C\delta'<\de$, the proposition follows. \end{proof}

The next step consists on showing that, for any large enough integer
$k$, we can find a complex-valued spherical harmonic $Y$ on $\SS^n$ with
eigenvalue $k(n+k-1)$ that approximates, in the ball $\BB_{1/k}$ and when rescaled, the function $\vphi$ in the unit ball. The proof hinges on
asymptotic expansions of ultraspherical polynomials, and uses the representation of~$\vphi$ as sum of shifted
Bessel functions which we just obtained as a key ingredient.

\begin{proposition}\label{P.spharm2}
Given constant $\de>0$, for any large enough positive integer $k$ there is a complex spherical harmonic $Y$ on
$\SS^n$ with energy~$k(n+k-1)$ satisfying
\begin{equation*}%\label{prop2e1}
\bigg\|\vphi-Y\circ \Psi^{-1}\Big(\frac\cdot k\Big)\bigg\|_{C^{m+2}(B)}<\de\,.
\end{equation*}
\end{proposition}

Proposition~\ref{P.spharm1} follows from this proposition, provided $k$ is large enough and $\de$ is
chosen small enough for $2\de$ not to be larger than~$\de'$. 

%\emph{Proof of Proposition \ref{P.spharm2}}:
\begin{proof}
Consider the ultraspherical polynomial of dimension
$n+1$ and degree~$k$, $C^{n}_{k}(t)$, which is defined as
\begin{equation}\label{cla}
C^{n}_k(t):=\frac{\Gamma(k+1) \Gamma(\frac{n}{2})}{\Gamma(k+\frac{n}{2})}\,P_{k}^{(\frac{n}{2}-1,\,\frac{n}{2}-1)}(t)\,,
\end{equation}
where $\Gamma(t)$ is the gamma function and $P_{k}^{(\alpha,\,\beta)}(t)$ are the Jacobi polynomials (see e.g~\cite[Chapter IV, Section 4.7]{Szego75}). We have included a normalizing factor so that $C^{n}_k(1)=1$ for all $k$ .

Let $p,q$ be two points in $\SS^{n}$, considered as the
set $\{|p|=1\}$ of $\RR^{n+1}$. The addition theorem for ultraspherical
polynomials ensures that $C^{n}_k(p\cdot q)$ (where $p\cdot q$
denotes the scalar product in~$\RR^{n+1}$ of the vectors $p$ and $q$) can be written as
\begin{equation}\label{clasp}
C^{n}_{k}(p\cdot q)=\frac{2\pi^{\frac{n+1}{2}}}{\Gamma(\frac{n+1}{2})} \frac{1}{d(k, n)}\sum_{j=1}^{d(k, n)}Y_{k j}(p)\,Y_{k j}(q)\,,
\end{equation}
with $\{Y_{k j}\}_{j=1}^{d(k, n)}$ being an arbitrary orthonormal
basis of spherical harmonics with eigenvalue $k(k+n-1)$. 

We recall that the function $\vphi$ was expressed as the finite sum 
\[
\vphi(x)=\sum_{j=1}^N c_j\,\frac{1}{|x-x_j|^{\frac{n}{2}-1}}J_{\frac{n}{2}-1}(|x-x_j|)\,,
\]
with coefficients $c_j\in \Cone$ and points $x_j \in B_R$. With these $c_j$ and $x_j$ we define, for any $p\in\SS^n$, the function
\begin{equation*}
Y(p):=\sum_{j=1}^N c_j \frac{1}{2^{\frac{n}{2}-1}\Gamma(\frac{n}{2})} \, C^{n}_k(p\cdot p_j)\,,
\end{equation*}
where $p_j:=\Psi^{-1}(\frac{x_{j}}{k})$. As long as $k>R$, $p_j$ is well defined. In view of Equation \eqref{clasp} it is clear that $Y$ is a spherical harmonic with eigenvalue $k(n+k-1)$.

Our aim is to study the asymptotic properties of the spherical harmonic
$Y$. To begin with, note that if we consider points $p:=\Psi^{-1}(\frac{x}{k})$ with $k>R$ and $x\in B_R$, we have
\begin{equation}\label{cos}
p\cdot p_j=\cos\big(\text{dist}_{\SS^n}(p,p_j)\big)=\cos \bigg(\frac{|x-x_j|+O(k^{-1})}{k}\bigg)\,,
\end{equation}
as $k\to \infty$. The last equality comes from $\Psi:\BB\to B$ being normal geodesic coordinates (by $\text{dist}_{\SS^n}(p,p_j)$ we mean the
distance between $p$ and $p_j$ considered on the sphere
$\SS^n$).  From now on we set
\begin{equation}\label{defi}
\tY(x):=Y\circ\Psi^{-1}\bigg(\frac{x}{k}\bigg)\,.
\end{equation}

When $k$ is large, one has
$$
\frac{\Gamma(k+1)}{\Gamma(k+\frac{n}{2})}=k^{1-\frac{n}{2}}+O(k^{-\frac{n}{2}})\,,
$$
so from Equation \eqref{cos} we infer
\begin{equation}\label{aaa}
C^{n}_k(p\cdot
p_j)=\bigg(\Gamma \Big(\frac{n}{2}\Big)k^{1-\frac{n}{2}}+O(k^{-\frac{n}{2}})\bigg)\,
P_{k}^{(\frac{n}{2}-1,\,\frac{n}{2}-1)}\bigg(
\cos\bigg(\frac{|x-x_j|+O(k^{-1})}{k}\bigg)\bigg)\,.
\end{equation}
By virtue of Darboux's formula for the Jacobi polynomials~\cite[Theorem 8.1.1]{Szego75}, we have the estimate
\begin{equation*}
\frac{1}{k^{\frac{n}{2}-1}}P_{k}^{(\frac{n}{2}-1,\,\frac{n}{2}-1)}\Big(\cos\frac{t}{k}\Big)=
2^{\frac{n}{2}-1}\, \frac{J_{\frac{n}{2}-1}(t)}{t^{\frac{n}{2}-1}}+O(k^{-1})\,,
\end{equation*}
uniformly in compact sets (e.g., for $|t|\leq
2R$). Hence, in view of Equation \eqref{aaa}, $\tY$ can be written as
\begin{align*}
\tY(x)&=\sum_{j=1}^N c_j \frac{1}{2^{\frac{n}{2}-1}\Gamma(\frac{n}{2})} C^{n}_k\Big(\cos\Big(\frac{|x-x_j|+O(k^{-1})}{k}\Big)\Big)\\&=\sum_{j=1}^{N}c_j \,\frac{1}{|x-x_j|^{\frac{n}{2}-1}}J_{\frac{n}{2}-1}(|x-x_j|)+O(k^{-1})\,,
\end{align*}
for $k$ big enough and $x, x_j\in B_R$. From this we get the
uniform bound
\begin{equation}\label{c0}
\|\vphi-\tY\|_{C^0(B)}<\de'\,
\end{equation}
for any $\delta'>0$ and all $k$ large enough.

It remains to promote this bound to the $C^{m+2}$ estimate. For this, note that, as the spherical harmonic $Y$  has eigenvalue $k(n+k-1)$, the rescaled function~$\tY$ verifies on $B$ the equation
\begin{equation*}
\Delta \tY+\tY=\frac{1}{k}A\tY\,,
\end{equation*}
with
$$
A\tY:=-(n-1)\tY+G_1\, \partial \tY+G_2\, \partial^2\tY\,,
$$
and where $\partial \tY$ (resp. $\partial^{2}\tY$) is a matrix whose entries are first (resp. second) order derivatives of $\tY$, and $G_i(x, k)$ are smooth matrix-valued functions with uniformly bounded derivatives, i.e.,
\begin{equation}\label{bg2a}
\sup_{x\in B}|\partial^\alpha_x G_i(x, k)|\leq C_{\al}\,.
\end{equation} 
with constants $C_{\al}$ independent of $k$.

Since $\vphi$ satisfies the Helmholtz equation $\De \vphi+\vphi=0\,$, the difference $\vphi-\tY$ satisfies
$$
\De(\vphi-\tY)+(\vphi-\tY)=\frac{1}{k}A\tY\,,
$$
and, considering the estimates~\eqref{c0} and~\eqref{bg2a}, by standard elliptic estimates we get
\begin{align*}
\|\vphi-\tY\|_{C^{m+2,\alpha}(B)}&<C\|\vphi-\tY\|_{C^0(B)}+\frac{C}{k}\|A\tY\|_{C^{m,\alpha}(B)}\\
&<C\delta'+\frac{C}{k}\|\vphi-\tY\|_{C^{m+2,\al}(B)}+\frac{C}{k}\|\vphi\|_{C^{m+2,\alpha}(B)}\,,
\end{align*}
so we conclude that, for $k$ big enough and $\de'$
small enough,
$$
\|\vphi-\tY\|_{C^{m+2}(B)}\leq C\de'+\frac{C \|\vphi\|_{C^{m+2,\al}}}{k}<\de
$$
and the proposition follows. \end{proof}

\section{Flexibility in euclidean space: proof of Theorem~\ref{T.euclidean}}
\label{S.euclidean}
Let us recall the setting: inside the unit $n$-dimensional ball $\BB \subset \RR^{n}$, we are given a collection $\mathfrak{S}:=\{ \Sigma_a \}_{a=1}^{r(n)}$ of $r(n)$ closed, pairwise disjoint, codimension 2 submanifolds. Our aim in this section is to find functions $\phi_a: \RR^{n}\rightarrow \Cone$, with $a=1,...,r(n)$, and a diffeomorphism $\Phi: \BB \rightarrow \BB$ , as close to the identity in the $C^{m}$ norm as we want, such that, on the one hand, each $\Phi(\Sigma_a)$ is a structurally stable nodal set of the corresponding $\phi_a$, and on the other hand, the spinor $\phi:=(\phi_1,...,\phi_{r(n)})$ satisfies the standard Dirac equation in $\RR^{n}$
\[
D_0 \phi= \sum_{\mu=1}^n \rho (e_\mu) \partial_\mu \phi=\lambda \phi 
\]
where $\{ e_\mu \}_{\mu=1}^{n}$ is an orthonormal basis on $\RR^{n}$ and $\rho(e_\mu)$ denotes the standard Clifford multiplication. We will henceforth fix $\lambda=1$, the general case being completely analogous.

To find and $\phi$ and $\Phi$, we will adapt the strategy introduced in \cite{EP} for analogous realization problems in second order elliptic PDEs. To begin with, one finds local solutions to the PDE under consideration that vanish at the $\Sigma_a$: this is achieved by means of some local existence theorem (we will use Cauchy-Kovalevskaya theorem). Then, one promotes these local solutions to global ones by means of a Runge-type approximation theorem (we will use Lax-Malgrange theorem). Through the process, one must ensure that the zero sets are structuraly stable, and this is done by prescribing the normal derivatives of the local solutions at the nodal sets and applying Thom's isotopy lemma.

Still, the application of the above scheme to our problem requires some additional considerations, the reason being, first of all, that the Dirac operator is of order 1, and thus we loose the capacity to directly prescribe the normal derivatives on the Cauchy problem; and further, that the Dirac operator mixes all components $\phi_a$ of $\phi$, so that we cannot solve the problem for each component separately. 
%We will deal with these peculiarities by posing a secondary extension problem, in which we will just make use of the extension properties of analytic functions. 

Let us begin with our construction. First of all, we note that, to use the Cauchy-Kovalevskaya theorem to find local solutions, our submanifolds $\Sigma_a$ should be real analytic.  This can be achieved by an arbitrarily small perturbation of the original submanifolds (because of the density of analytic functions in the space of smooth functions). Indeed, since the normal bundle of each $\Sigma_a$ is of rank 2, it is always trivial by a theorem of W. S. Massey \cite{MA}, and therefore we can always find a smooth complex-valued function $F_a$ vanishing at $\Sigma_a$ and whose differential has full rank at $\Sigma_a$. %with $F_a^{-1}(0)=\Sigma_a$, and rank$(d_x F_a)=2$ for all $x\in \Sigma_a$. 
The full rank condition makes the zero set $\Sigma_a$ of $F_a$ structurally stable, so a real analytic function approximating $F_a$ well enough is guaranteed to have as zero set a very small perturbation of $\Sigma_a$ (in the sense of being diffeomorphic by an ambient diffeomorphism arbitrarily close to the identity).

We will keep denoting by $\Sigma_a$ the analytic submanifolds obtained after small perturbation, and by $F_a$ the collection of complex-valued, real analytic functions realizing such submanifolds as zero sets.

For each $a$, let $M_a$ be an analytic hypersurface containing $\Sigma_a$, and such that, for $a\neq b$, $M_a$ and $M_b$ are disjoint. For example, we can choose $M_a$ to be
\[
M_a:=\{ x \in \RR^{n}, \text{dist}(x, \Sigma_{a})\leq \epsilon_a, \text{Im}(F_a)(x)=0\}
\]
for $\epsilon_a$ small enough.

%On each $M_a$ we define a real analytic, complex valued function $\tphi_a: M_a\rightarrow \Cone$ by setting $\tphi_a(x)=F_a(x)$ for all $x\in M$. 

We will denote by $n_a$ the normal vector field to $M_a \subset \RR^{n}$, and by $\nu_a$ the normal vector field to $\Sigma_a \subset M_a$, when considering $\Sigma_a$ as a hypersurface of $M_a$. 

The next lemma will play an important role in what is to follow, providing the initial condition for the Cauchy problem that is compatible with structural stability:

\begin{lemma} \label{stability} There is a collection of real-analytic functions $g_b: \Sigma_a\rightarrow \Cone$, with $b \in \{1,...,r(n)\}\setminus \{a\}$, such that the spinor $\vphi:=(g_1, g_2,...,\nu_a \cdot \nabla F_a,...,g_{r(n)})$ verifies the condition
\[
\text{Im } [\rho(n_a)\rho(\nu_a)\vphi(x)]_a > 0
\]
for all $x \in \Sigma_a$. Here, by $[\cdot]_a$ we denote the $a$-th component of a spinor. 
\end{lemma}

\begin{proof} Let $\Gamma(x)=\{\Gamma_{cd}(x)\}$, with $c,\,d=1,...,r(n)$, denote the matrix $\rho(n_a)\rho(\nu_a)(x)$. Note that at each point $x \in \Sigma_a$, $\Gamma(x)$ is an antihermitian matrix, so $\Gamma_{aa}(x)$ is purely imaginary or zero. The above condition can be written as

\begin{equation}\label{positive}
-i \Gamma_{aa}(x) \nu_a \cdot \nabla \text{Re} (F_a)+\text{Im} \sum_{b=1}^{r(n)} \Gamma_{ab}(x) g_b(x)>0.
\end{equation}

The matrix $\Gamma(x)$ always has non-zero determinant, so the functions $\Gamma_{ab}(x)$, for $b=1,...,r(n)$, cannot all vanish at once. Besides, note that the term $\nu_a \cdot \nabla \text{Re} F_a(x)$ cannot vanish either, because rank$(dF_a |_{\Sigma_a})$ is 2 while $\nabla \text{Im} F_a(x)|_{M_a}$ vanishes. Thus, at any point $x_0\in \Sigma_a$, it is easy to see that one can find a collection of complex numbers $w_{b0}$ such that condition \ref{positive} holds. Further, since all the coefficients in condition \ref{positive} are analytic functions, and the condition is open, on a sufficiently small neighborhood of the point $x_0$ we can promote the complex numbers $w_{b0}$ to complex valued real analytic functions $g_{b0}(x)$ satisfying \ref{positive}, with $g_{b0}(x_0)=w_{b0}$. 

Thus, the above discussion grants the existence of an open cover $\{U_\alpha\}$ of $\Sigma_a$ and an associated collection of complex analytic functions $g_{b \alpha}: U_\alpha \rightarrow \Cone$ satisfying \ref{positive} on each $U_\alpha$. Therefore, if $\{ \psi_\alpha, U_\alpha\}$ is a partition of unity, it is easy to check that the smooth functions
\[
g_b(x):=\sum_{\alpha} \psi_\alpha(x) g_{b \alpha} (x)
\]
satisfy the condition \ref{positive} for all $x \in \Sigma_a$. These functions fail to be analytic, because they are defined through a partition of unity, however, since the condition is open and $\Sigma_a$ is compact, we can find analytic functions close enough to the $g_b$ so that \ref{positive} is still satisfied.   
\end{proof}

Now, let $N_a\subset M_a$ be a tubular neighbourhood of $\Sigma_a$ (to define the tubular neighborhood, $\Sigma_a$ is being considered as hypersurface of $M_a$), and let $X$ be the projection of the vector field $\nabla \text {Re} (F_a)$ to the tangent space of $M_a$. Provided $N_a$ is small enough, the vector field $X$ is never zero on $N_a$. Thus, any point $y \in N_a$ can be written uniquely as $y=\Phi^{t}_X(x)$, for some $x\in \Sigma_a$ and $t\in\RR$. The functions $f_b$ on $N_a$ defined as
\[
f_b(y)=f_b(\Phi^{t}_X (x))= t g_b(x), 
\]
are real-analytic, vanish on $\Sigma_a$, and verify that $\nu_a \cdot \nabla f_b =g_b$. With these functions as input, we pose the following Cauchy problem for the Dirac operator on a neighborhood of the hypersurfaces $N_a$:
\begin{equation}\label{initial value 2}
\begin{cases}
    
D_0 \phi =\phi \\
 \phi_b |_{N_a}=f_b \text{ for $b \in \{1,...,r(n)\}\setminus \{a\}$}\\
 \phi_a |_{N_a}=F_a

\end{cases}
\end{equation}

The Cauchy-Kovalevskaya theorem yields a solution $\phi$ to \eqref{initial value 2} on a small enough tubular neighborhood $U_a$ of $N_a$ (it is understood that we have one such solution $\phi^{a}$ for each submanifold $\Sigma_a$, but we will not make this explicit for ease of notation). 

Let us analyze the properties of this local solution. First of all, $\Sigma_a$ is a nodal set of the component $\phi_a$ of $\phi$, because we have imposed that $\phi_a |_{\Sigma_a}=F_a|_{\Sigma_a}=0$. Further, by virtue of the properties of the functions $f_b$, we are going to see that this nodal set is structurally stable (note that, in fact, $\Sigma_a$ is also a nodal set of the rest of the components $\phi_b$, but it need not be structurally stable for them).

To get structural stability, we need to ensure that the real and imaginary parts of $\nabla \phi_a$ are linearly independent on $\Sigma_a$. Firstly, note that we can write
$$
\nabla \phi_a \rvert_{\Sigma_a} = (n_a \cdot \nabla \phi_a) n_a + \nabla_{N_a} \phi_a |_{\Sigma_a} = (n_a \cdot \nabla \phi_a)  n_a + X |_{\Sigma_a}
$$
because the projection of $\nabla \text{Im}(F_a)|_{\Sigma_a}$ into the tangent space of $N_a$  is zero. In addition, $X$ is real and does not vanish. Hence, for $\nabla \text{Re} (\phi_a)$ and $\nabla \text{Im} (\phi_a)$ to be linearly independent, it is necessary and sufficient that the term $\text{Im }(n_a \cdot \nabla \phi_a)$ does not vanish on $\Sigma_a$. 

Next, note that the Dirac equation on $\Sigma_a$ can be written as
\[
n_a\cdot \nabla \phi|_{\Sigma_a}=\rho(n_a)\rho(\nabla_M \phi)|_{\Sigma_a},
\]
because $\phi |_{\Sigma_a}=0$ and $\rho(n_a)\rho(n_a)=-\text{Id}$. It is at this point that the properties of the functions $f_b$ are important: since $f_b |_{\Sigma_a}=0$ and $\nu_a \cdot\nabla f_b=g_b$, we have that
\[
\rho(n_a)\rho(\nabla_M \phi)|_{\Sigma_a}=\rho(n_a)\rho(\nu_a)\vphi
\]
with $\vphi=(g_1,...,\nabla \text{Re}(F_a),...,g_{r(n)})$. Hence, by Lemma \ref{stability}, the term
 $$
 \text{Im }[\rho(n_a)\rho(\nu_a)\vphi ]_a=\text{Im } n_a\cdot \nabla \phi|_{\Sigma_a}
 $$
 is greater than zero, and structural stability follows.

To sum up, what we get after carrying out the same construction for each submanifold $\Sigma_a$ is a collection of local solutions $\{\phi^{a}=(\phi^{a}_1,..., \phi^{a}_{r(n)})\}_{a=1}^{r(n)}$ to the Dirac equation, each one being defined on a small enough open neighborhood $U_a$ of the corresponding submanifold $\Sigma_a$ (small enough so that the open sets can be assumed to be pairwise disjoint) and verifying that $\Sigma_a=(\phi^{a}_{a})^{-1}(0)$, with the structural stability property. Since the complement of the union of all the open sets $U_a$ has no relatively compact components, the Lax-Malgrange approximation theorem \cite{NA} ensures the existence, for any given constant $\delta>0$ and integer $m\geq1$, of a global solution $\phi'$ to the Dirac equation on $\RR^{n}$ satisfying

\[
\|\phi^{a}-\phi'\|_{C^{m}(U_a)}\leq \delta.
\] 

In particular, the components verify $\|\phi^{a}_a-\phi'_a\|_{C^{m}(U_a)}\leq \delta$. Since $\text{rank}(d\phi'_a) |_{\Sigma_a}$ is 2, by Thom's isotopy lemma we can choose an appropriate $\delta$ so that there is a diffeomorphism $\Phi: \BB \rightarrow \BB$ verifying $\|\Phi-\text{Id}\|_{C^{m}(U_a)}\leq \epsilon$, for $\epsilon$ as small as desired, and such that $\Phi(\Sigma_a) \subset \phi'^{-1}_{a}(0)$.

\section{Prescribing nodal sets at different regions at once}
\label{S.lens}
To set the stage, consider a collection of structures $\{ \mathfrak{S}^{\alpha} \}_{\alpha=1}^\La$, where $\La$ is a positive integer that we can choose as large as we want, and where each $\mathfrak{S}^{\alpha}$ is itself a collection of $r(n)$ codimension 2 submanifolds, $\mathfrak{S}^{\alpha}=\{\Sigma_1^{\alpha},...,\Sigma_{r(n)}^{\alpha}\}$. 

Following the notation in Section \ref{S.main}, for a set of points $\{p_\alpha \}_{\alpha=1}^\La$ in $\SS^{n}$ we denote by $\Psi_{\alpha} : \BB_{\rho}(p_\alpha) \rightarrow B_\rho$ the corresponding geodesic patches on balls centered on the $p_{\alpha}$ of radius $\rho$. We fix a radius $\rho$ such that no two balls intersect, for example by setting 
\[
\rho:=\frac12\min_{\alpha \neq \beta}\dist_{\SS^n}(p_\alpha, p_\beta)\,.
\]

Our goal is to find a Dirac spinor realizing each structure $\mathfrak{S}^{\alpha}$ within the ball $\BB_{\rho}(p_\alpha)$. 

\begin{remark} Note that it is necessary to choose the set $\{p_\alpha\}_{\alpha=1}^\La$ so that no pair of points are antipodal in $\SS^{n}\subset \RR^{n+1}$, i.e, $p_\alpha\neq-p_\beta$ for all $\alpha$, $\beta$. The reason is that if a structure $\mathfrak{S}$ is realized by the Dirac spinor $\psi$ of eigenvalue $\frac{n}{2}+k$ in the ball $\BB_{k^{-1}}(p)$, $\mathfrak{S}$ is also automatically realized by $\psi$ in the antipodal ball $\BB_{k^{-1}}((-1)^{k}p)$. Indeed, the components $\psi_a$ have parity $(-1)^{k}$, that is
\[
\psi_a (-p)=(-1)^{k} \psi_a (p),
\]
because, being complex spherical harmonics of energy $k(n+k-1)$, they are the restriction to the sphere of a couple of real harmonic homogenous polynomials of degree $k$.
\end{remark}

The realization of many different structures at once is a direct consequence of the proposition that is to follow. It exploits the fast decay of ultraspherical polynomials of high degree outside the balls where they behave as shifted Bessel functions:

\begin{proposition}\label{P.spharm.corr}
Let $\{\phi_\alpha\}_{\alpha=1}^\La$ be a set of $\La$ complex-valued functions in $\RR^n$, satisfying
$\Delta \phi_\alpha+\phi_\alpha=0$. Fix a positive integer $m$ and positive constant $\de$.  For any large enough integer $k$, there is a spherical harmonic $Y$ on
$\SS^n$ with energy $k(n+k-1)$ verifying the bound
\begin{equation*}%\label{prop2e1}
\bigg\|\phi_\alpha-Y\circ \Psi_{\alpha}^{-1}\Big(\frac\cdot k\Big)\bigg\|_{C^{m+2}(B_{\rho})}<\de\,
\end{equation*}
for all $1 \leq \alpha \leq \La$. 
\end{proposition}

\begin{proof} Applying Proposition \ref{P.spharm1} to each $\phi_\alpha$ we obtain, for high enough $k$, complex spherical harmonics $\{ Y_\alpha \}_{\alpha=1}^{\La}$ satisfying the bound
\begin{equation*}%\label{prop2e1}
\bigg\|\phi_\alpha-Y_\alpha\circ \Psi_{\alpha}^{-1}\Big(\frac\cdot k\Big)\bigg\|_{C^{m+2}(B)}<\de\,.
\end{equation*}
The spherical harmonics $Y_\alpha(p)$ are linear combinations of ultraspherical polynomials  $C^{n}_{k}(p\cdot q)$, with $\dist_{\SS^n}(p_\alpha, q)$ proportional to $k^{-1}$. Recall that ultraspherical polynomials verified the asymptotic formula
\[
C^{n}_k(p\cdot
q)=\frac{\Gamma(\frac{n}{2})}{k^{\frac{n}{2}-1}}\,
P_{k}^{(\frac{n}{2}-1,\,\frac{n}{2}-1)}(
\cos(\dist_{\SS^n}(p, q)))+O(k^{-\frac{n}{2}})\,,
\]
and considering the fact that the Jacobi polynomials behave, uniformly for $k^{-1}<t<\pi-k^{-1}$, as (see ~\cite[Theorem 7.32.2]{Szego75})
\[
k^{\frac{n}{2}-1}\, P_k^{(\frac{n}{2}-1,\frac{n}{2}-1)}(\cos t)=\frac{O(k^{-1})}t,
\]
we can conclude that the $C^{n}_k(p \cdot q)$ are uniformly
bounded as
\[
|C^{n}_k(p\cdot q)|\leq \frac{C_\rho}{k}
\]
for any points $p$ and $q$ verifying
\begin{equation*}%\label{distpq}
\dist_{\SS^n}(p,q)\geq \rho \quad \text{and}\quad \dist_{\SS^n}(p,-q)\geq \rho \,,
\end{equation*}
and where $C_{\rho}$ is a constant depending only on $\rho$.
The same decay is thus also exhibited by the spherical harmonics $Y_\alpha$, 
\[
\|Y_{\alpha}\|_{C^0(\SS^n\backslash (\BB(p_\alpha,\rho)\cup \BB(-p_\alpha,\rho))}\leq \frac{C'_\rho}{k}
\]
since they are just normalized linear combinations of ultraspherical polynomials (here $C'_{\rho}$ depends also on the particular coefficients in the expansion of $Y_\alpha$, that is, on $\phi_\alpha$ and $\delta$). 

Now, if we define the spherical harmonic
\[
Y:=\sum_{\alpha=1}^\La Y_{\alpha}\,
\]
and we choose $k$ large enough, and $\rho$ small enough so that the sets $\BB(p_\alpha,\rho)\cup \BB(-p_\alpha,\rho)$ are disjoint for all $\alpha$, the desired bound automatically follows in the $C^{0}$ norm. By standard elliptic estimates, we promote it to the $C^{m+2}$ norm, and the proposition follows. \end{proof}

Given this proposition, arguing exactly as in Sections 2 and 3 we can find a spinor $\psi$ and a diffeomorphism $\Phi$ such that $\Phi(\Sigma^{\alpha}_a)\subset\psi^{-1}_a(0) \cap \BB(p_\alpha, \rho)$, for $\alpha=1,...,\La$. 

%As a conclusion to this section, let us note that this construction can be used straightforwardly to obtain an analog of Theorem \ref{T.main1} in some sphere quotients, by building appropriately equivariant eigenfunctions of the Dirac operator with prescribed nodal sets. The reader is referred to \cite{EPT} for a more detailed account of this procedure in the case of curl eigenfields. 

\section{Concluding remarks: the torus case}\label{S:final}

We conclude by sketching how to prove an analog of Theorem \ref{T.main1} on the torus $\TT^{n}$. 

The spin structures on the $n$-dimensional torus $\TT^{n}=\RR^{n}/\ZZ^{n}$ are indexed by elements in the first cohomology class $H^{1}(\TT^{n}, \ZZ_{2})\cong\ZZ^{n}_{2}$. The Dirac operator corresponding to the zero cohomology class can be regarded as the one inherited from the standard Dirac operator on $\RR^{n}$.  

For this Dirac operator, an inverse localization theorem (as Theorem \ref{T.approx1}) can be proven following a similar strategy as the one in Section 3. To begin with, Weitzenb\H{o}ck formulas (and hence the relation between spinors and eigenfunctions of the Laplace-Beltrami operator) are trivial. Next, for the proof of the torus analog of Proposition \ref{P.spharm1}, it suffices to argue as in Section \ref{S.spharm1}, with the following adaptation: the role of the shifted Bessel functions centered at points $x_j$ is now played by trigonometric polynomials of frequencies $\xi_j$, with $|\xi_j|=1$; and the role of the ultraspherical polynomials is now played by trigonometric polynomials of frequencies $k\xi_j$ (with $k$ an eigenvalue of the Dirac operator in the torus), with the further requirement that $k$ and $\xi_j$ verify $k\xi_j \in \ZZ^{n}$ (so that a trigonometric polynomial of frequencies $k\xi_j$ defines an eigenfunction of the laplacian on the torus). These results in hand, the reasoning in Section \ref{S.main} can be directly transplanted to the torus case.

Interestingly enough, a number-theoretical subtlety arises when carrying out the above scheme. In the torus analog of Proposition \ref{P.Bessel1}, one would like to approximate the complex-valued function $\phi_1$ (in the notation of Section \ref{S.spharm1})
\[
\phi_1(x)=\int_{\SS^{n-1}}f_1(\xi)\, e^{ix\cdot\xi}\, d\si(\xi)\,,
\] 
by a sum of the form
\[
\vphi(x)=\sum_{j=1}^{N} c_j e^{i\xi_j \cdot x}
\]
with $|\xi_j|=1$ and $k\xi_j \in \ZZ^{n}$ for high enough $k$. It is obvious that one can always approximate the above integral as much as one wishes by a discretized sum centered at some points $\xi_j \in \SS^{n}$; the problem is that we have a very specific requirement on the points $\xi_j$. The approximation would hold whenever, for a sequence of eigenvalues  $k\rightarrow \infty$ going to infinity, there are corresponding sets of points $\{ \xi_n \}^{N}_{n=1}$ with $k\xi_n \in \ZZ^{n}$ that become dense in $\SS^{n-1}$.

However, not all sequences of eigenvalues might verify this property. As a matter of fact, if we restrict our attention to sequences of integer eigenvalues, and we substitute the density property by the (stronger) property of equidistribution of the corresponding sequences of sets of rational points, the problem is related to the celebrated Linnik problem in number theory (when $k$ is an integer, we say that the frequencies $\xi_j$ with $\xi_j k\in \ZZ^{n}$ are rational points on $\SS^{n-1}\subset \RR^{n}$ of \emph{height} $k$). For instance, in dimension $3$, only certain sequences of integers are known to verify the equidistribution property \cite{Du88, Du03}, for example, sequences of odd integers, or sequences of integers whose squares are square-free. To take this phenomenon into account, one has to modify accordingly the statement of Theorem \ref{T.main1}, so that it reflects that it is the eigenfunctions with, say, sufficiently high \emph{odd} eigenvalue (in the case of $\TT^{3}$) that realize the given structure $\mathfrak{S}$.

\section*{Acknowledgments}

The author is grateful to Daniel Peralta-Salas for his valuable suggestions, especially concerning section \ref{S.euclidean}, and to \'Angel David Mart\'inez for his critical reading of the manuscript. The author is supported by the ERC Starting Grant~335079 and by a fellowship from Residencia de Estudiantes. He also thanks the Department of Mathematics of Harvard University, where part of this work was carried out, for its hospitality. This work is supported in part by the ICMAT--Severo Ochoa grant SEV-2015-0554.

\bibliographystyle{amsplain}

\end{document}